\documentclass[11pt]{amsart}
\usepackage{amsfonts}
\usepackage{amssymb}
\usepackage{amsmath}
\usepackage[table]{xcolor}
\usepackage{hyperref}
\usepackage{color}
\usepackage{cite}
\usepackage{graphicx}
\usepackage{epsfig}
\definecolor{anti-flashwhite}{rgb}{0.95, 0.95, 0.96}
\definecolor{gray(x11gray)}{rgb}{0.75, 0.75, 0.75}
\theoremstyle{plain}
\newtheorem{Thm}{Theorem}[section]

\newtheorem{Lem}[Thm]{Lemma}

\newtheorem{Rem}[Thm]{\sl Remark}

\theoremstyle{definition}
\newtheorem{Def}[Thm]{Definition}

\theoremstyle{remark}

\errorcontextlines=0 \numberwithin{equation}{section}

\begin{document}
\title[Fractional Banach Spaces]
      {Characterizations of Compact Operators on $\ell_{p}$ Type Fractional Sets of Sequences }

\author{Faruk \"{O}zger}

\address[\"{O}zger]{Department of Engineering Sciences, \.{I}zmir Katip \c{C}elebi University, Izmir, Turkey}
\email[\"{O}zger]{farukozger@gmail.com}


\begin{abstract}
Among the sets of sequences studied, difference sets of sequences are probably the most common type of sets. This paper considers some $\ell_{p}$ type fractional difference sequence spaces via Euler gamma function. Although we characterize compactness conditions on those spaces using the main tools of Hausdorff measure of noncompactness, we can only obtain sufficient conditions when the final space is $\ell _{\infty }$. However, we use some recent results to exactly characterize the classes of compact matrix operators when the final space is the set of bounded sequences.
\end{abstract}

\maketitle

\section{Introduction}
The Euler gamma function of a real number $ x $ (except zero and the negative integers) is defined by an improper integral:

\begin{eqnarray*}
\Gamma \left ( x \right )=\int_{0}^{\infty}e^{-t}t^{x-1}dt.
\end{eqnarray*}
\\
It is known that for any natural number  $n$, $\Gamma(n+1)=n!$, and $\Gamma(n+1)=n\Gamma(n)$ holds for any real number $n\notin\left \{ 0,-1,-2,... \right \}$.

The fractional difference operator for a fraction $\tilde{\alpha}$ have been defined in \cite{BalDut} as

\begin{eqnarray}\label{del}
\Delta ^{\left (  \tilde{\alpha} \right )}(x_{k})=\sum_{i=0}^{\infty}\left ( -1 \right )^{i}\frac{\Gamma \left ( \tilde{\alpha} +1 \right )}{\Gamma \left ( \tilde{\alpha} -i+1 \right )}x_{k-i}.
\end{eqnarray}

It is assumed that the series defined in (\ref{del}) is convergent for $x\in\omega$.

Let $m$ be a positive integer, then recall the difference operators $\Delta^{(1)}$ and $ \Delta ^{(m)}$ by:

\begin{eqnarray*}
(\Delta^{(1)}x)_{k}=\Delta^{(1)}x_{k}=x_{k}-x_{k-1}
\end{eqnarray*}
and
\begin{eqnarray*}
(\Delta^{(m)}x)_{k}=\sum_{i=0}^{m}(-1)^{i}\binom{m}{i}x_{k-i}.
\end{eqnarray*}

We write $\Delta$ and $\Delta^{(m)}$ for the matrices with $\Delta_{nk}=(\Delta^{(1)}e^{(k)})_{n}$ and
$\Delta_{nk}^{(m)}=(\Delta^{(m)}e^{(k)})_{n}$
for all $n$ and $k$. The topological properties of some spaces that are constructed by the matrix operator $\Delta^{(m)}$ were studied in the paper \cite{ColEt}. In  \cite{djolovic2}, some identities and estimates for the Hausdorff measure of noncompactness of matrix operators from $\Delta^{(m)}$ type spaces into the sets of bounded, convergent, null sequences and also absolutely convergent series were established.

We also write fractional difference operator as an infinite matrix:
\begin{equation*}
\Delta ^{\left (  \tilde{\alpha} \right )}_{nk} =\left\{
\begin{array}{lll}
(-1)^{n-k}\frac{\Gamma (\tilde{\alpha}+1)}{(n-k)!\Gamma (\tilde{\alpha}-n+k+1)} &  & (0\leq k\leq n) \\
0 &  & (k>n).%
\end{array}%
\right.
\end{equation*}

\begin{Rem}
The inverse of fractional difference matrix is given by
\begin{equation*}
\Delta ^{\left (  -\tilde{\alpha} \right )}_{nk} =\left\{
\begin{array}{lll}
(-1)^{n-k}\frac{\Gamma (-\tilde{\alpha}+1)}{(n-k)!\Gamma (-\tilde{\alpha}-n+k+1)} &  & (0\leq k\leq n) \\
0 &  & (k>n).%
\end{array}%
\right.
\end{equation*}
\end{Rem}
For some values of $\tilde{\alpha}$, we have
\begin{eqnarray*}
\Delta ^{1/2}x_{k}&=&x_{k}-\frac{1}{2}x_{k-1}-\frac{1}{8}x_{k-2}-\frac{1}{16}x_{k-3}-\frac{5}{128}x_{k-4}-...\\
\Delta ^{-1/2}x_{k}&=&x_{k}+\frac{1}{2}x_{k-1}+\frac{3}{8}x_{k-2}+\frac{5}{16}x_{k-3}+\frac{35}{128}x_{k-4}+...\\
\Delta ^{2/3}x_{k}&=&x_{k}-\frac{2}{3}x_{k-1}-\frac{1}{9}x_{k-2}-\frac{4}{81}x_{k-3}-\frac{7}{243}x_{k-4}-...
\end{eqnarray*}

The idea of constructing new sequence spaces via infinite matrices started with K\i zmaz's study \cite{kızmaz} and then it has been developed by numerous researchers using different triangles \cite{kaoz}, \cite{p15}, \cite{ema3}, \cite{ema}, \cite{FatEr}.

In the studies \cite{kaoz2}, \cite{AyBas}, \cite{fofb}, \cite{fofb2}, \cite{PolAl}, \cite{candan}, \cite{kirisci} \cite{Muretal} and \cite{Talo} different difference sequence spaces have been studied based on some newly defined infinite matrices. Some new results on the visualization and animations of the topologies of certain sequence spaces spaces have been illustrated in \cite{mo}, \cite{mov2}, \cite{mov3}.  The authors applied their software package for this purpose. Note that, those results have an interesting and important application in crystallography.

Many authors have made efforts to apply Hausdorff measure of noncompactness to find compactness conditions of certain sets of sequences during the past decade  \cite{ema2}, \cite{mo}, \cite{karabasmur} and \cite{murnom}.  

Fractional difference sequence spaces have been studied in the literature recently  \cite{Fur}, \cite{BalDut}, \cite{KadBal}. The authors of those papers especially studied on the properties of fractional operators in their researches in addition to focusing on certain fractional sequence spaces.

In this work, we consider the fractional sequence spaces $ \ell_{p} (\Delta ^{(  \widetilde{\alpha} )}  ) $ for $ 1\leq p< \infty $ and determine norm operators of our spaces. We establish some identities or estimates for the Hausdorff measures of noncompactness of certain operators on difference sequence spaces $ \ell_{p} (\Delta ^{(  \widetilde{\alpha} )}  ) $  of fractional orders. We characterize some classes of compact operators on those spaces. Note that, we can only obtain sufficient conditions when the final space is $\ell _{\infty }$. However, we use the results in \cite{katarina} and \cite{Sargent} to obtain necessary and sufficient conditions for the classes of compact matrix operators from $\ell_{p} (\Delta ^{(  \widetilde{\alpha} )}  )$ spaces into the sets of bounded sequences and for the classes $(\ell_{1} (\Delta ^{(  \widetilde{\alpha} )}  ), \ell_{\infty })$ and $(\ell _{\infty }(\Delta ^{(  \widetilde{\alpha} )}  ),\ell _{\infty })$.

\newpage

\section{Preliminaries}

The $\beta$ dual of a set $X$ is defined by

\begin{equation*}
X^{\beta}=\{a\in \omega:a\cdot x\in cs \mbox{ for all }x\in X\}.
\end{equation*}

Note that $c_{0}^{\beta}=c^{\beta}=\ell_{\infty}^{\beta}=\ell_{1}$ and $\ell_{p}^{\beta}=\ell_{q}$.

Given any infinite matrix $A=(a_{nk})_{n,k=0}^{\infty }$ of complex
numbers and any sequence $x$, we write $A_{n}=(a_{nk})_{k=0}^{\infty
}$ for the
sequence in the $n^{th}$ row of $A$, $A_{n}x=\mbox{$\sum_{k=0}^{\infty}$}
a_{nk}x_{k}$ $(n=0,1,\dots )$ and $Ax=(A_{n}x)_{n=0}^{\infty }$, provided
$A_{n}\in X^{\beta }$ for all $n$.

If $X$ and $Y$ are subsets of
$\omega $, then
\begin{equation*}
X_{A}=\{x\in \omega :Ax\in X\}
\end{equation*}

denotes the matrix domain of $A$ in $X$
and $(X,Y)$ is the class of all infinite matrices that map $X$ into $Y$; so $A\in (X,Y)$ if and only if $X\subset Y_{A}$.

Consider now the following fractional difference sequence spaces for $1\leq p <\infty$:

\begin{eqnarray*}
\ell_{p} (\Delta ^{(  \widetilde{\alpha} )}  )&:=&\left \{x=(x_{k})\in\omega: \sum_{n=0}^{\infty}\left\vert \sum_{i=0}^{\infty}(-1)^{i}\frac{\Gamma(\tilde{\alpha} +1)}{i!\Gamma(\tilde{\alpha} -i+1)}x_{n-i}
\right\vert^{p}<\infty\right \}.
\end{eqnarray*}

Now let us define the sequence $y =(y _{k})$ which will be
used, by the $\Delta ^{(  \tilde{\alpha} )}$-transform of a sequence $
x=(x_{k}) $, that is,

\begin{eqnarray*}
y_{k}&=& x_{k}-\tilde{\alpha} x_{k-1}+\frac{\tilde{\alpha} (\tilde{\alpha} -1)}{2!}x_{k-2}-\frac{\tilde{\alpha} (\tilde{\alpha} -1)(\tilde{\alpha} -2)}{3!}x_{k-3}+\cdots  \\
&=&\sum_{i=0}^{\infty}(-1)^{i}\frac{\Gamma(\tilde{\alpha} +1)}{i!\Gamma(\tilde{\alpha} -i+1)}x_{k-i}. \label{yk}
\end{eqnarray*}

Hence, those spaces can be considered as the matrix domains of the triangle $\Delta ^{(  \tilde{\alpha} )}$ in the classical sequence spaces $\ell_{p}$, where $1\leq p <\infty$. We also have the following relation between the sequences $x=(x_{k}) $ and $y=(y_{k}) $:

\begin{eqnarray*}
x_{k}=\sum_{i=0}^{\infty}(-1)^{i}\frac{\Gamma(-\tilde{\alpha} +1)}{i!\Gamma(-\tilde{\alpha} -i+1)}y_{k-i}.
\end{eqnarray*}

\begin{Lem}
\label{BKspace} \cite[Theorem 4.3.12, p. 63]{Wil2} Let
$\left( X,\left\Vert .\right\Vert \right) $ be a $ BK $ space. Then
$X_{T}$\ is a $ BK $ space with $\left\Vert .\right\Vert _{T}=\left\Vert
T\left( .\right) \right\Vert .$
\end{Lem}

By Lemma \ref{BKspace}, defined fractional difference sequence space is a complete, linear, $BK$ space with the following norm:

\begin{eqnarray*}
\left\Vert x\right\Vert=\left(\sum\limits_{n}\left\vert  \sum_{i=0}^{\infty}(-1)^{i}\frac{\Gamma(\tilde{\alpha} +1)}{i!\Gamma(\tilde{\alpha} -i+1)}x_{n-i}
\right\vert^{p}\right)^ \frac{1}{p}.
\end{eqnarray*}

Let $X$ be a normed space. Then $S_{X}=\{x\in X:\left\Vert x\right\Vert =1\}$ and $\bar{B}_{X}=\{x\in
X:\left\Vert x\right\Vert \leq 1\}$ denote the unit sphere and closed unit ball in $X$, where $X$ is a normed space. 

By $\mathcal{F}_{r}  $ $ (r=0,1,\dots) $, we denote the subcollection of $\mathcal{F}$ consisting of all nonempty and finite subsets of $ \mathbb{N}$ with terms that are greater than than $ r $, that is 

\[
\mathcal{F}_{r} =\left\lbrace N \in \mathcal{F}_{r}: n>r \textsl{    for all  } n \in N \right\rbrace  \ (r=0,1,\dots).
\]

Given $a\in \omega $, we write
\begin{equation*}
	\left\Vert a\right\Vert _{X}^{\ast }=\sup\limits_{x\in
		S_{X}}\left\vert \sum\limits_{k=1}^{\infty }a_{k}x_{k}\right\vert
\end{equation*}%
provided the expression on the write hand side is defined and finite
which is the case whenever $X$ is a $BK$ space and $a\in X^{\beta }$.

\begin{Lem}\label{lem1.3}
Let $X$ and $Y$ be $BK$ spaces.
\begin{itemize}
\item[(i) ] Then we have $\left( X,Y\right)\subset\mathcal{B}\left( X,Y\right)$, that is,
every  $A\in\left( X,Y\right)$ defines an operator
$L_{A}\in\mathcal{B}\left( X,Y\right)$, where $L_{A}(x)=Ax$ for all
$x\in X$ (see \cite[Theorem 1.23]{mara}).
\item[(ii) ] We have $\Vert x \Vert_{X}^{*}=\Vert x \Vert_{X^\beta}$ for all $x\in X^\beta$, where $\Vert \cdot \Vert_{X^\beta}$ is the natural norm on the dual set $X^\beta$ (see \cite[Theorem 3.2]{Mara7}).
 \end{itemize}
\end{Lem}

\begin{Lem}\label{operatornorm.1}
	Let $Y$ be an arbitrary subset of $\omega$ and $X$ be a $BK$ space with AK or $X=\ell_{\infty}$, and $R=S^{t} $. Then $A\in(X_{T},Y)$ if and only if $\hat{A}\in\left( X,Y\right)$ and $W^{(n)}\in(X,c_{0})$ for all $n=0,1,\ldots $. Here $\hat{A}$ is the matrix with rows $\hat{A}_{n}=RA_{n}$ for $n=0,1,\dots$, and the triangles $W^{(n)}$ $(n=0,1,\ldots)$ are defined by $ w_{mk}^{(n)} =\sum_{j=m}^{\infty}a_{nj}s_{jk}$.
	
	Moreover, if $A\in\left( X,Y\right)$ then we have $Az=\hat{A}(Tz)$ for all $z\in Z=X_{T}$ (see \cite{Mara7}, Theorem 3.4, Remark 3.5(a)).
\end{Lem}

We have the following results the for operator norms of bounded operators by \cite[Theorem 2.8]{djolovıc}.
\begin{Lem}\label{lem1.4}
Let $X$ be a $BK$ space.
\begin{itemize}
\item[(i) ] If $A\in (X,\ell_{1})$, then
\[
\Vert A\Vert_{(X,\ell_1)}\leq \Vert L_{A}\Vert \leq 4\cdot \Vert A\Vert_{(X,\ell_1)},
\]
where 
\[\Vert A\Vert_{(X,\ell_1)}=\sup_{N \in \mathcal{F}} \left\|  \sum_{n\in N} A_n\right\| _{X}^{*} < \infty.\]

\item[(ii) ] If $A\in (X,Y)$, then
\[
\Vert L_{A}\Vert=\Vert A \Vert _{(X,\ell_{\infty}) }=\sup_{n}\Vert A_n \Vert_{X }^* < \infty,
\]
where $Y$ is any of the spaces $c_0,c,\ell_\infty$.
\end{itemize}
\end{Lem}

We obtain the following result as an immediate consequence of Lemma \ref{lem1.3}(ii).
\begin{Lem}\label{lem6.3} Let $1\leq p<\infty$, then we have  $\Vert x \Vert_{\ell_{p} }^*=\Vert x \Vert_{\ell_{q} }$ for all $x=(x_k)\in \ell_{q}$.
\end{Lem}

\begin{Rem}
	Let $1\leq p< \infty$, $A$ be an infinite matrix and $X$ be $BK$ space. If $A \in \left( \ell_{p} (\Delta ^{(  \widetilde{\alpha} )}  ),Y\right)$, then $\hat{A} \in (\ell_{p}, Y) $ such that $Ax=\hat{A}y$ for all  $x \in \ell_{p} (\Delta ^{(  \widetilde{\alpha} )}  )$ and $y \in \ell_{p}  $. Here $x$ and $y$ connected by (\ref{yk}) and $\hat{A} =(\hat{a}_{nk}) $ is defined by
	\begin{equation}
	\hat{a}_{nk}=\sum\limits_{j=k}^{\infty
	}(-1)^{j-k}\frac{\Gamma(-\tilde{\alpha} +1)}{(j-k)!\Gamma(-\tilde{\alpha} +j-k+1)}a_{nj}\text{		for all }n,k\in \mathbb{N}_{0}\text{;} \label{Ahat}
	\end{equation}
	
	and for $k=0,1,\ldots $ we also define  $\hat{\alpha}=(\hat{\alpha}_{k})_{k=0}^{\infty }$ by
	
	\begin{equation} \label{alphak}
	\hat{\alpha}_{k}=\lim_{n\rightarrow \infty }\hat{a}_{nk}. 
	\end{equation}
\end{Rem}

\begin{proof}
	Let $1\leq p< \infty$, $A \in \left( \ell_{p} (\Delta ^{(  \widetilde{\alpha} )}  ),Y\right)$ and $ x \in \ell_{p} (\Delta ^{(  \widetilde{\alpha} )}  )$. So, $ A_n \in (\ell_{p} (\Delta ^{(  \widetilde{\alpha} )}  ))^\beta$ for all $n=0,1,\ldots $. Then, we have $\hat{A}_n \in \ell_{p}^{\beta}= \ell_{q}$ for all $n=0,1,\ldots $ and $Ax=\hat{A}y$ satisfies. Therefore, $\hat{A}y \in Y$ and this means $\hat{A} \in (\ell_{p}, Y) $.
\end{proof}

\begin{Lem}\label{lem6.4} If $a=(a_k)\in \left( \ell _{p}(\Delta ^{(  \tilde{\alpha} )}) \right)^\beta$, then $\overline{a}=(\overline{a}_k)\in \ell_{q}$
and the equality

\begin{equation}
\sum_{k=0}^{\infty}a_k x_k = \sum_{k=0}^{\infty}\overline{a}_k y_k . \label{akxk}
\end{equation}

satisfies for every $x=(x_k)\in  \ell _{p}(\Delta ^{(  \tilde{\alpha} )}) $, where
\begin{equation}
\overline{a}_k= \sum_{i=k}^{\infty}(-1)^{i-k}\frac{\Gamma(-\tilde{\alpha} +1)}{(i-k)!\Gamma(-\tilde{\alpha} -i+k+1)}a_{i}.\label{abar}
\end{equation}
\end{Lem}

Let us now write $S=S_{\ell_{p }}$ and $\hat{S}=S_{\ell_{p }(\Delta ^{(  \tilde{\alpha} )})}$, for short.

\begin{Lem}\label{lem6.5} Let $1\leq p<\infty$ and $\overline{a}=(\overline{a}_k)$ be defined in (\ref{abar}), then we have

\begin{equation*}
\Vert a \Vert_{\ell _{p}(\Delta ^{(  \tilde{\alpha} )}) }^* =\Vert \overline{a} \Vert_{\ell_{q} }=\left\{
\begin{array}{lll}
\left(\sum\limits_{k=0}^{\infty}|\overline{a}_k|^q \right)^\frac{1}{q}&  & (1< p\leq \infty) \\
\sup\limits_{k}|\overline{a}_k| &  & (p=1).%
\end{array}%
\right.
\end{equation*}
for all $a=(a_k)\in \left( \ell _{p}(\Delta ^{(  \tilde{\alpha} )}) \right)^\beta$.
\end{Lem}

\begin{proof}
Assume that $\overline{a}\in \left( \ell _{p}(\Delta ^{(  \tilde{\alpha} )}) \right)^\beta$. So, we have by Lemma \ref{lem6.4} that $\overline{a}\in \ell_q$ and the equality (\ref{akxk}) yields for all sequences $x=(x_k)\in \left( \ell _{p}(\Delta ^{(  \tilde{\alpha} )}) \right)$ and $y=(y_k)\in \ell_p$ which are connected by (\ref{yk}). Additionally, we have by Lemma \ref{lem6.3} that $x\in \hat{S}$ if and only if $y\in S$. Hence, we have by (\ref{akxk}) that

\begin{equation} \label{lpdeltanorm}
\Vert a \Vert_{\ell _{p}(\Delta ^{(  \tilde{\alpha} )}) }^* = \sup_{x\in \hat{S}}\left|\sum_{k=0}^{\infty}a_k x_k\right|=
\sup\limits_{y\in S}\left|\sum_{k=0}^{\infty}\overline{a}_k y_k \right| =\Vert \overline{a} \Vert^{*}_{\ell_{p} } < \infty.
\end{equation}

Morover, we obtain by (\ref{lpdeltanorm}) and Lemma \ref{lem6.3} that
\begin{equation*}
\Vert a \Vert_{\ell _{p}(\Delta ^{(  \tilde{\alpha} )}) }^* =\Vert \overline{a} \Vert^{*}_{\ell_{p} } =\Vert \overline{a} \Vert_{\ell_{q} }
\end{equation*}
because $\overline{a}\in \ell _{q}$. This completes the proof.
\end{proof}

We obtain estimates or identities for the norms of the matrix operators for the classes $(\ell_{p} (\Delta ^{(  \widetilde{\alpha} )}  ),c _0)$, $(\ell_{p} (\Delta ^{(  \widetilde{\alpha} )}  ),c )$, $(\ell _{p }(\Delta ^{(  \widetilde{\alpha} )}  ),\ell _{\infty })$ and $ (\ell_{p} (\Delta ^{(  \widetilde{\alpha} )}  ),\ell_{1})$.

\begin{Thm}\label{thm2.5}
Let $A$ be in any of the classes $(\ell_{p} (\Delta ^{(  \widetilde{\alpha} )}  ),c _0)$, $(\ell_{p} (\Delta ^{(  \widetilde{\alpha} )}  ),c)$ or
$(\ell _{p }(\Delta ^{(  \widetilde{\alpha} )}  ),\ell _{\infty })$, then
\[
\Vert L_{A}\Vert=\Vert A \Vert _{(\ell_{p} (\Delta ^{(  \widetilde{\alpha} )}  ),\ell_{\infty}) },
\]

where 

\[
\Vert A \Vert _{(\ell_{p} (\Delta ^{(  \widetilde{\alpha} )}  ),\ell_{\infty}) }=\sup_{n}\left( \sum\limits_{k=0}^{\infty} \left|  \hat{a}_{nk}  \right|^q \right)^{\frac{1}{q}} < \infty
\]
\end{Thm}

\begin{proof}
This is an immediate consequence of Lemma \ref{lem1.3} and Lemma \ref{lem1.4}.
\end{proof}

\begin{Thm}\label{thm2.6}
Let $A\in (\ell_{p} (\Delta ^{(  \widetilde{\alpha} )}  ),\ell_{1})$, then
\[
\Vert A\Vert_{(\ell_{p} (\Delta ^{(  \widetilde{\alpha} )}  ),\ell_1)}\leq \Vert L_{A}\Vert \leq 4\cdot \Vert A\Vert_{(\ell_{p} (\Delta ^{(  \widetilde{\alpha} )}  ),\ell_1)},
\]
where 

\[
\Vert A\Vert_{(\ell_{p} (\Delta ^{(  \widetilde{\alpha} )}  ),\ell_1)}=\sup_{N \in \mathcal{F}}\left(  \sum\limits_{k=0}^{\infty}  \left| \sum_{n\in N} \hat{a}_{nk} \right|   \right) < \infty.
\]
\end{Thm}

\begin{proof}
This is an immediate consequence of Lemma \ref{lem1.3} and Lemma \ref{lem1.4}.
\end{proof}

\section{Main Results Related to Compact Operators}

The following notations are needed to establish estimates and
identities for the Hausdorff measure of noncompactness of matrix
operators and characterize the classes of compact operators. We also use the results in Katarina's paper \cite{katarina} to prove our results.

We recall the definition of the Hausdorff measure of noncompactness
of bounded subsets of a metric space, and the Hausdorff measure of
noncompactness of operators between Banach spaces.

If $X$ and $Y$ are infinite--dimensional complex Banach spaces then
a linear operator $L:X\rightarrow Y$ is said to be compact if the
domain of $L$ is
all of $X$, and, for every bounded sequence $(x_{n})$ in $X$, the sequence $%
(L(x_{n}))$ has a convergent subsequence. We denote the class of
such operators by $\mathcal{C}(X,Y)$.

\begin{Def}
Let $(X,d)$ be a metric space, $B(x_{0},\delta )=\{x\in
X:d(x,x_{0})<\delta \}$ denote the open ball of radius $\delta >0$
and center in $x_{0}\in X$, and $\mathcal{M}_{X}$ be the collection
of bounded sets in $X$. The Hausdorff measure of noncompactness of
$Q\in \mathcal{M}_{X}$ is
\begin{equation*}
\chi (Q)=\inf \{\epsilon > 0:Q\subset
\bigcup_{k=1}^{n}B(x_{k},\delta _{k}):x_{k}\in X,\ \delta
_{k}<\epsilon,\ 1\leq k \leq n,\ n\in\mathbb{N}\}.
\end{equation*}%
Let $X$ and $Y$ be Banach spaces and $\chi _{1}$ and $\chi _{2}$ be
measures of noncompactness on $X$ and $Y$. Then the operator
$L:X\rightarrow Y$ is called $(\chi _{1},\chi _{2})$--bounded if
$L(Q)\in \mathcal{M}_{Y}$ for every $Q\in \mathcal{M}_{X}$ and there
exists a positive constant $C$ such that
\begin{equation}
\chi _{2}(L(Q))\leq C\chi _{1}(Q)\mbox{ for every }Q\in
\mathcal{M}_{X}. \label{EmaFarAb.S.3.Eq.1}
\end{equation}%
If an operator $L$ is $(\chi _{1},\chi _{2})$--bounded then the
number
\begin{equation*}
\Vert L\Vert _{(\chi _{1},\chi _{2})}=\inf \left\{ C\geq 0:(\ref%
{EmaFarAb.S.3.Eq.1})\mbox{ holds for all }Q\in
\mathcal{M}_{X}\right\}
\end{equation*}%
is called the $(\chi _{1},\chi _{2})$--measure of noncompactness of
$L$. In particular, if $\chi _{1}=\chi _{2}=\chi $, then we write
$\Vert L\Vert _{\chi }$ instead of $\Vert L\Vert _{(\chi ,\chi )}$.
\end{Def}

\begin{Lem} \cite[ Corollary 2.26.]{mara}
\label{compactoperator}Let $X$ and $Y$ are Banach spaces and $L\in \mathcal{B}\left( X,Y\right) $. Then we have
\begin{equation}\label{hmnLS}
\left\Vert L\right\Vert _{\chi }=\chi \left( L(\bar{B}_{X})\right)
=\chi \left( L(S_{X})\right), 
\end{equation}
\begin{equation}
L\in \mathcal{C}(X,Y)\text{ if and only if }\left\Vert L\right\Vert
_{\chi }=0.  \label{c14}
\end{equation}
\end{Lem}

\begin{Lem} \cite[ Theorem 2.23.]{mara}
Let $X$ be a Banach space with Schauder basis $\left( b_{n}\right) _{n=0}^{\infty }$, $%
Q\in \mathcal{M}_{X},$ $P_{n}:X\rightarrow X$ be the projector onto
the linear span of $\left\{ b_{0},b_{1},\ldots b_{n}\right\} $. $I$
be the identity map on $X$ and $R_{n}=I-P_{n}$ $(n=0,1,\dots)$.
Then we have%
\begin{equation}
\dfrac{1}{a}\cdot \limsup_{n\rightarrow \infty } \left( \sup_{x\in
Q}\left\Vert  R_{n} (x)\right\Vert \right) \leq \chi (Q)\leq
\limsup_{n\rightarrow \infty } \left( \sup_{x\in Q}\left\Vert R_{n}
(x)\right\Vert \right), \label{c16}
\end{equation}%
where $a=\limsup_{n\rightarrow \infty } \left\Vert
R_{n}\right\Vert$.
\end{Lem}

\begin{Lem} \cite[ Theorem 2.8.]{mara}
Let $Q$\ be a bounded subset of the normed
space $X$, where $X$ is $\ell _{p}$\ for $1\leq p<\infty $ or $c_{0}$. If\ $%
P_{n}:X\rightarrow X$ is the operator defined by $P_{n}(x)=x^{[n]}$ for $%
x=(x_{k})_{k=0}^{\infty }\in X$, then we have

\begin{equation*}
\chi (Q)=\lim_{n}\left( \sup_{x\in Q}\left\Vert R_{n} (x)\right\Vert
\right).
\end{equation*}
\end{Lem}

When the first space is $ \ell_{1 }$, a problem takes place since $ \beta $ dual of   $\ell_{1 }$ is  $ \ell_{\infty }$, which has no $ AK $. In this case, the following study of Sargent \cite{Sargent}  is used  in order to characterize compact operators.

\begin{Lem}\label{lem2.8}
\cite[ Theorem 5.]{Sargent} If $L\in B(\ell_1,\ell_\infty)$, then $L$
is compact if and only if
\[
\lim_{m\to\infty}\sup_{1\leq n\leq
m}|a_{n,k_1}-a_{n,k_2}|=\sup_n|a_{n,k_1}-a_{n,k_2}|
\]
uniformly in $k_1$ and $k_2$, $(1\leq k_1,k_2<\infty)$.
\end{Lem}

The following results (by \cite[Corollary 3.6]{djolovıc} and \cite[Theorems 3.7 and 3.11]{murnom}) are needed to determine estimates for the norms of continuous linear operators $ L_A $ on our spaces and establish necessary and sufficient conditions for a matrix operator to be a compact operator.

\begin{Lem} \label{Laksi}Then we have:
	\begin{itemize}
		\item[(i) ] If $A\in \left(X,c_{0} \right)$, then we have
		\begin{equation*}
		\Vert L_{A}\Vert _{\chi } = \limsup_{n\rightarrow \infty } \Vert A_n\Vert_{X}^{*}.     
		\end{equation*}
	
		\item[(ii) ] If $A\in \left( X,c\right)$, then we have
		\begin{equation*}
		\dfrac{1}{2} \cdot \limsup_{n\rightarrow \infty } \Vert A_n-\alpha \Vert_{X}^{*}
		\leq \Vert L_{A}\Vert _{\chi } \leq \limsup_{n\rightarrow \infty } \Vert A_n-\alpha\Vert_{X}^{*}.
		\end{equation*}
	
		\item[(iii) ] If $A\in \left( X,\ell_{1}\right)$, then we have
		\begin{equation*}
		\lim_{r\rightarrow \infty } \left( \sup_{N \in \mathcal{F}_{r}} \left\| \sum_{n\in N} A_n \right\| _{X}^{*}   \right) 
		\leq \Vert L_{A}\Vert _{\chi } \leq 4\cdot \lim_{r\rightarrow \infty } \left( \sup_{N \in \mathcal{F}_r} \left\|   \sum_{n\in N} A_n \right\| _{X}^{*}\right).
		\end{equation*}
	\end{itemize}		
\end{Lem}

We are now ready to state the main results related to compact operators. We start with establishing some estimates for the norms of bounded linear operators $ L_A $ on the given fractional sequence spaces  $ \ell_{p}(\Delta ^{(  \tilde{\alpha} )})$.

\begin{Thm}\label{main}
Let $1<p<\infty$ and $q=\frac p{p-1}$.

\begin{itemize}
\item[(i) ] If $A\in \left( \ell _{p}(\Delta ^{(  \tilde{\alpha} )}),c_{0} \right)$, then we have

\begin{equation} \label{m1}
\Vert L_{A}\Vert _{\chi }= \lim_{r\to\infty}  \sup_{n\geq r} \left(\sum_{k} |\hat{a}_{nk}|^q\right)^{\frac
1q}. 
\end{equation}

\item[(ii) ] If $A\in \left( \ell _{p}(\Delta ^{(  \tilde{\alpha} )}),c\right)$, then we have
\begin{eqnarray}\label{m2}
\dfrac{1}{2}\cdot \lim\limits_{r\rightarrow \infty }\sup\limits_{n\geq r}\left(
\sum\limits_{k}\left\vert \hat{a}_{nk}-
\hat{\alpha}_{k}\right\vert^q\right)^{\frac
1q}
&\leq &\Vert L_{A}\Vert _{\chi }\leq \lim\limits_{r\rightarrow \infty }\sup\limits_{n\geq r}\left(
\sum\limits_{k}\left\vert \hat{a}_{nk}-
\hat{\alpha}_{k}\right\vert^q\right)^{\frac
1q}.  
\end{eqnarray}

\item[(iii) ] If $A\in \left( \ell _{p}(\Delta ^{(  \tilde{\alpha} )}),\ell_{1}\right)$, then we have

\begin{equation}\label{m3}
\lim_{r\rightarrow \infty } \left( \sup_{N \in \mathcal{F}_r} \left(
\sum\limits_{k}\left\vert \sum_{n\in N}
\hat{a}_{nk}\right\vert^q\right)^{\frac
	1q}   \right) 
\leq \Vert L_{A}\Vert _{\chi }\leq 4\cdot \lim_{r\rightarrow \infty } \left( \sup_{N \in \mathcal{F}_r} \left(
\sum\limits_{k}\left\vert \sum_{n\in N}
\hat{a}_{nk}\right\vert^q\right)^{\frac
	1q}   \right) .
\end{equation}

\end{itemize}
\end{Thm}

\begin{proof}
Let $A\in \left( \ell _{p}(\Delta ^{(  \tilde{\alpha} )}),c_{0}\right)$. By Lemma \ref{lem6.5}, we have  

\begin{equation*}
\Vert A\Vert _{\ell _{p}(\Delta ^{(  \tilde{\alpha} )})}^{*}=\Vert \hat{A}_n\Vert _{\ell _{q}}=\left(\sum_{k=0}^{\infty} |\hat{a}_{nk}|^q\right)^{\frac
	1q}. 
\end{equation*}

for all $n=0,1,\ldots $ because $A_n \in \left( \ell _{p}(\Delta ^{(  \tilde{\alpha} )}) \right)^{\beta}$ for all $n=0,1,\ldots $. Taking into account this result and Lemma \ref{Laksi}(i) we obtain (\ref{m1}). 

In order to prove the second part we start with $A\in \left( \ell _{p}(\Delta ^{(  \tilde{\alpha} )}),c\right)$. By Remark \ref{Ahat} $\hat{A}\in \left( \ell _{p},c\right)$. So, we have by Lemma \ref{Laksi}(ii) that
\begin{equation*}
	\dfrac{1}{2} \cdot \limsup_{n\rightarrow \infty } \Vert \hat{A}_n-\hat{\alpha} \Vert_{\ell _{q}}
	\leq \Vert L_{A}\Vert _{\chi } \leq \limsup_{n\rightarrow \infty } \Vert \hat{A}_n-\hat{\alpha}\Vert_{\ell _{q}}, 
\end{equation*}
where $\hat{\alpha}$ defined in (\ref{alphak}). Hence, by Lemma \ref{compactoperator} we get

\begin{equation} \label{LAS1}
\left\Vert  L_{A}\right\Vert _{\chi }=\chi \left( L_{A}(\hat{S})\right)
=\chi \left( A\hat{S}\right),  
\end{equation}
and
\begin{equation} \label{LAS2}
\left\Vert  L_{\hat{A}}\right\Vert _{\chi }=\chi \left( L_{\hat{A}}(S)\right)
=\chi \left( \hat{A}S\right).  
\end{equation}
Morover, $ x \in \hat{S} $ if and only if $ y \in S $. We have $A\hat{S}=\hat{A}S $ because $ Ax=\hat{A}y $ by Remark \ref{Ahat}. Taking into account the equalities (\ref{LAS1}) and (\ref{LAS2}) we conclude that $\left\Vert  L_{A}\right\Vert _{\chi }= \left\Vert  L_{\hat{A}}\right\Vert _{\chi } $ which completes the second part of the proof.

Finally, let $A\in \left( \ell _{p}(\Delta ^{(  \tilde{\alpha} )}),\ell_{1}\right)$. We derive from Lemma \ref{lpdeltanorm} that 
\begin{equation}\label{l1}
\left\| \sum_{n\in N} A_n \right\| _{\ell _{p}(\Delta ^{(  \tilde{\alpha} )})}^{*} =
\left\| \sum_{n\in N} \hat{A}_n \right\| _{\ell _{q}}
\end{equation}

because $ A_n \in (\ell _{p}(\Delta ^{(  \tilde{\alpha} )}))^\beta $. Hence, we obtain the condition in (\ref{m3}) by (\ref{l1}) and Lemma \ref{Laksi}(iii).

\end{proof}

\begin{Thm}
Let $1\leq p<\infty$.

\begin{itemize}
\item[(i) ] If $A\in \left( \ell _{1}(\Delta ^{(  \tilde{\alpha} )}),c_{0}\right)$ then we have

\begin{equation}
\Vert L_{A}\Vert _{\chi }= \lim_{r\to\infty}  \sup_{n\geq r} \left(\sup_{k} | \hat{a}_{nk}|\right).
\end{equation}

\item[(ii) ] If $A\in \left( \ell _{1}(\Delta ^{(  \tilde{\alpha} )}),c\right)$ then we have
\begin{eqnarray}
\dfrac{1}{2}\cdot \lim\limits_{r\rightarrow \infty }\sup\limits_{n\geq r}\left(
\sup_{k}\left\vert \hat{a}_{nk}-\hat{\alpha}_{k}\right\vert\right)
&\leq &\Vert L_{A}\Vert _{\chi }\leq \lim\limits_{r\rightarrow \infty }\sup\limits_{n\geq r}\left(
\sup_{k}\left\vert \hat{a}_{nk}-\hat{\alpha}_{k}\right\vert\right).  \notag
\end{eqnarray}

\end{itemize}
\end{Thm}

\begin{proof}
This is an immediate consequence of Theorem \ref{thm2.5}, Lemma \ref{Laksi} and Theorem \ref{main}.
\end{proof}

We now use the results in \cite{katarina} and \cite{Sargent}  to obtain necessary and sufficient conditions for the classes of compact matrix operators in $(\ell_{p} (\Delta ^{(  \widetilde{\alpha} )}  ),\ell _{\infty })$, $(\ell_{1} (\Delta ^{(  \widetilde{\alpha} )}  ),\ell _{\infty })$ and
$(\ell _{\infty }(\Delta ^{(  \widetilde{\alpha} )}  ),\ell _{\infty })$.

\begin{Thm} Let $1<p<\infty$ and $q=\frac p{p-1}$.
\begin{itemize}
\item[(i) ] If $A\in (\ell _{p}(\Delta ^{(  \tilde{\alpha} )}),\ell_\infty)$,
then $L_A$ is compact if and only if 

\[ \lim_{r\rightarrow \infty
}\sup_{n}\left(\sum_{k=r+1}^{\infty }|\hat{a}_{nk}|^{q}\right)^{\frac{1}{q}}=0 .
\]

\item[(ii) ] If $A\in (\ell_1(\Delta ^{(  \tilde{\alpha} )}),\ell_\infty )$ , then $L_A$ is compact if and only if
\[
\lim_{m\to\infty}\sup_{1\leq n\leq m}|\hat {a}_{n,k_1}-\hat{a}
_{n,k_2}|=\sup_n|\hat{a}_{n,k_1}-\hat{a}_{n,k_2}|
\]
uniformly in $k_1$ and $k_2$, $(1\leq k_1,k_2<\infty)$ .
\end{itemize}
\end{Thm}

\begin{proof}
Assume $A\in (\ell _{p}(\Delta ^{(  \tilde{\alpha} )}),\ell _{\infty })$. By Lemma \ref{operatornorm.1}, we obtain $L_{A}\in \mathcal{B}(\ell _{p}(\Delta ^{(  \tilde{\alpha} )}),\ell _{\infty })$ and $\hat{A}\in (\ell _{p},\ell _{\infty })$
such that $Ax=\hat{A}(\Delta ^{(  -\tilde{\alpha} )}x)$ for all $x\in \ell _{p}(\Delta ^{(  \tilde{\alpha} )})$, where $\Delta ^{( - \tilde{\alpha} )}$ is the inverse of the triangle $\Delta ^{(  \tilde{\alpha} )}$ and $\hat{a}_{nk}$ defined in (\ref{Ahat}). We have $L_{\hat{A}}\in \mathcal{B}(\ell _{p},\ell _{\infty })$ and {$\hat{A}x=L_{\hat{A}}(x)$} for all $x\in \ell _{p}$ because $\ell _{p}$ is a $BK$ space. If we write $(\hat{A}_{n})_{r}$ for the $r$th row of the matrix $\hat{A}$ with $\hat{a}_{nk}$ replaced by 0 for $k>r$,

\[
(\hat{A}_{n})_{r}=\sum_{k=0}^{r}\hat{a}_{nk}e^{(k)}
\]

and apply {\cite[(b), p.85 ]{Sargent}} we obtain $L_{\hat{A}}$ is compact if and only if

\[
\sup_{n}||\hat{A}_{n}||_{q}=\sup_{n}\left(\sum_{k}|\hat{a}_{nk}|^{q}\right)^{\frac{1}{
q}}<\infty
\]
and
\[ \lim_{r\rightarrow \infty }\sup_{n}||\hat{A}_{n}-(\hat{A}_{n})_{\mathbf{r}
}||_{q}=\lim_{r\rightarrow \infty
}\sup_{n}\left(\sum_{k=r+1}^{\infty }|\hat{a}_{nk}|^{q}\right)^{\frac{1}{q}}=0 .
\]

The first condition yields from the characterizations of matrix transformations, \cite[Example 8.4.5D and Example 8.4.6D]{Wil2}. Hence $L_{\hat{A}}$ is compact if and only if the condition

\[ \lim_{r\rightarrow \infty
}\sup_{n}\left(\sum_{k=r+1}^{\infty }|\hat{a}_{nk}|^{q}\right)^{\frac{1}{q}}=0
\]

yields. Taking into account the equalities (\ref{LAS1}) and (\ref{LAS2}) we have $L_{A}$ is compact if and only if $L_{\hat{A}}$ is compact. This completes the first part of the proof.

We know that $L_{A}$ is compact if and only if $L_{\hat{A}
}$ is compact with the equalities (\ref{LAS1}) and (\ref{LAS2}). Also $A\in (\ell_{1}(\Delta ^{(  \tilde{\alpha} )} ,\ell _{\infty })$ implies
$\hat{A}\in (\ell_{1},\ell _{\infty })$. As an immediate consequence of Lemma \ref{lem2.8} we prove the second part of the theorem.
\end{proof}

\begin{Thm}
Let $A\in (\ell_\infty (\Delta ^{(  \tilde{\alpha} )}),\ell_\infty)$. Then $L_A$ is compact if and only
if
\begin{equation}
\lim_{r\to\infty}  \sup_n \sum_{k=r+1}^\infty |\hat{a}_{nk}|=0. \label{linf}
\end{equation}
\end{Thm}

\begin{proof}
By the help of Lemma \ref{operatornorm.1} we have $\hat{A}\in (\ell
_{\infty },\ell _{\infty })$ because $\ell_{\infty }$ is a $ BK $ space and also by Lemma \ref{lem1.3} the associated bounded linear operator
$L_{\hat{A}}$. Since $ (\ell _{\infty })^{\beta}=\ell_{1}$, therefore by the condition (b) in \cite[p. 85]{Sargent} we have: The spaces
$\ell_\infty$ and $c_0$ have the same dual space $\ell_{1}$, and therefore
the conditions \cite[(b) and (f), p.85]{Sargent} are the same. It follows that
$L_{\hat{A}}$ is compact if and only if

\[
\sup_{n}||\hat{A}_{n}||_{1}=\sup_{n}\sum_{k=0}^{\infty }|\hat{a}_{nk}|<\infty
\]
and
\[
\lim_{r\rightarrow
\infty }\sup_{n}||\hat{A}_{n}-(\hat{A}_{n})_{\mathbf{r}}||_{1}=\lim_{r\rightarrow \infty }\sup_{n}\sum_{k=r+1}^{\infty
}|\hat{a}_{nk}|=0.
\]
So the first condition is satisfied \cite[ Example 8.4.5A]{Wil2},
and it follows that $L_{\hat{A}}$ is compact if and only if the condition (\ref{linf}) yields. This completes the proof because $L_{A}$ is compact if and only if $L_{\hat{A}}$ is compact by Lemma \ref{lem1.3} and the equalities (\ref{LAS1}) and (\ref{LAS2}).
\end{proof}

\section{Conclusion}
The main aim of this paper is to consider fractional sequence spaces $ \ell_{p} (\Delta ^{(  \widetilde{\alpha} )}  ) $ for $ 1\leq p< \infty $ via Euler gamma function. We derive some estimates and identities for the norms of bounded linear operators on fractional sequence spaces $ \ell_{p} (\Delta ^{(  \widetilde{\alpha} )}  ) $. In general, Hausdorff measure of noncompactness is used to  find necessary and sufficient conditions for a matrix operator on a given sequence space to be a compact operator. However we can only obtain sufficient conditions when the final space is $\ell _{\infty }$, that is for a matrix class $ (X,\ell _{\infty }) $. This is why we use the results in \cite{katarina} and \cite{Sargent} to exactly obtain necessary and sufficient conditions for the classes of compact matrix operators in $(\ell_{p} (\Delta ^{(  \widetilde{\alpha} )}  ),\ell _{\infty })$ and
$(\ell _{\infty }(\Delta ^{(  \widetilde{\alpha} )}  ),\ell _{\infty })$. In addition to these characterizations, we apply Hausdorff measure of noncompactness to establish necessary and sufficient conditions for a matrix operator to be a compact operator from fractional sequence spaces $\ell _{p }(\Delta ^{(  \widetilde{\alpha} )}   $ and $\ell _{1 }(\Delta ^{(  \widetilde{\alpha} )}  )$ into $ Y $, where $ Y $ is any of the spaces of convergent to zero, convergent and bounded sequences $c _{0 }, c, \ell _{1 }  $.

All these results are given with the following table:
\\

\begin{center} 
	\rowcolors{0.5}{anti-flashwhite}{gray(x11gray)}
	\begin{tabular}{||l|c|c||}\hline \hline
		\begin{tabular}[t]{lr}
			& From   \\
			To &
		\end{tabular}
		&  \ \ \ \ \ $\mathbf{\ell_{p} (\Delta ^{(  \widetilde{\alpha} )}  )}$ \ \ \ & \ \ \ \ \ $\mathbf{\ell_{1}(\Delta ^{(  \tilde{\alpha} )} )}$ \ \ \  \\ 
		$c_{0}$ & {\bf 1.} & {\bf 5.} \\
		$c$& {\bf 2.} & {\bf 6.} \\
		$\ell_{\infty}$ & {\bf 3.} & {\bf 7.} \\ 
		$\ell_{1}$ & {\bf 4.} & {\bf -} \\ \hline  
	\end{tabular}
\end{center}

\begin{center} 
	\textbf{Table.} Necessary and sufficient conditions for an operator to be compact 
\end{center}
\begin{itemize}
	\item[[\textbf{1.}] ]$ \lim\limits_{r\rightarrow \infty }\sup\limits_{n\geq r}\left(\sum\limits_{k=0}^{\infty}\left\vert \sum\limits_{j=k}^{\infty
	}(-1)^{j-k}\frac{\Gamma(-\tilde{\alpha} +1)}{(j-k)!\Gamma(-\tilde{\alpha} +j-k+1)}a_{nj}\right\vert^q \right)^{\frac 1q} =0. $ 
	\item[[\textbf{2.}] ]$ \lim\limits_{r\rightarrow \infty }\sup\limits_{n\geq r}\left(\sum\limits_{k=0}^{\infty}\left\vert \sum\limits_{j=k}^{\infty
	}(-1)^{j-k}\frac{\Gamma(-\tilde{\alpha} +1)}{(j-k)!\Gamma(-\tilde{\alpha} +j-k+1)}a_{nj}-\lim\limits_{n\rightarrow \infty }\hat{a}_{nk}\right\vert^q \right)^{\frac 1q} =0. $
	\item[[\textbf{3.}] ] $ \lim\limits_{r\rightarrow \infty
	}\sup\limits_{n}\left(\sum\limits_{k=r+1}^{\infty }\left\vert \sum\limits_{j=k}^{\infty
}(-1)^{j-k}\frac{\Gamma(-\tilde{\alpha} +1)}{(j-k)!\Gamma(-\tilde{\alpha} +j-k+1)}a_{nj}\right\vert^{q}\right)^{\frac{1}{q}}=0.$
	\item[[\textbf{4.}] ] $ \lim\limits_{r\rightarrow \infty }  \sup\limits_{N \in \mathcal{F}_r} \left(\sum\limits_{k=0}^{\infty}\left\vert \sum\limits_{n\in N}
	\sum\limits_{j=k}^{\infty
	}(-1)^{j-k}\frac{\Gamma(-\tilde{\alpha} +1)}{(j-k)!\Gamma(-\tilde{\alpha} +j-k+1)}a_{nj}\right\vert^q\right)^{\frac	1q} =0 . $
	\item[[\textbf{5.}] ] $ \lim\limits_{r\rightarrow \infty }\sup\limits_{n\geq r} \left(\sup\limits_{k} \left\vert \sum\limits_{j=k}^{\infty
	}(-1)^{j-k}\frac{\Gamma(-\tilde{\alpha} +1)}{(j-k)!\Gamma(-\tilde{\alpha} +j-k+1)}a_{nj}\right\vert\right)=0.$
	\item[[\textbf{6.}] ] $ \lim\limits_{r\rightarrow \infty }\sup\limits_{n\geq r}\left(\sup\limits_{k}\left\vert \sum\limits_{j=k}^{\infty
	}(-1)^{j-k}\frac{\Gamma(-\tilde{\alpha} +1)}{(j-k)!\Gamma(-\tilde{\alpha} +j-k+1)}a_{nj}-\lim\limits_{n\rightarrow \infty }\hat{a}_{nk}\right\vert\right)=0. $
	\item[[\textbf{7.}] ] $	\lim\limits_{m\to\infty}\sup\limits_{1\leq n\leq m}|\hat {a}_{n,k_1}-\hat{a} _{n,k_2}|=\sup\limits_n|\hat{a}_{n,k_1}-\hat{a}_{n,k_2}|$, \\
	uniformly in $k_1$ and $k_2$, $(1\leq k_1,k_2<\infty)$ .
\end{itemize}

\section*{Compliance with ethical standards}
\subsection*{Conflict of interest}
The authors declare that they have no conflict of
interest.

\subsection*{Ethical approval}
This article does not contain any studies with human
participants or animals performed by any of the authors.

\end{document}